\documentclass[oneside,final,11pt]{amsart}

\usepackage{dsfont}

\newcommand{\R}{\mathds R}
\newcommand{\absconv}{\mathrm{absconv}}

\newcommand{\EP}{$c_0$EP}

\newcommand{\upstar}[1]{{#1}^{\raise1pt\hbox{$\scriptscriptstyle*$}}}
\newcommand{\chilow}[1]{\chi_{\lower2pt\hbox{$\scriptstyle#1$}}}

\DeclareMathOperator{\Ker}{Ker}
\DeclareMathOperator{\Img}{Im}
\DeclareMathOperator{\supp}{supp}
\DeclareMathOperator{\height}{ht}

\title[On the $c_0$-extension property]{On the $\mathbf{c_0}$-extension property}

\author{Claudia Correa}
\thanks{The author was partially supported by FAPESP grant 2018/09797-2.}
\address{Centro de Matem\'atica, Computa\c c\~ao e Cogni\c c\~ao \hfill\break\indent Universidade Federal do ABC, Brazil}
\email{claudiac.mat@gmail.com, claudia.correa@ufabc.edu.br}
\urladdr{http://professor.ufabc.edu.br/\~{}claudia.correa}

\date{May 16th, 2020}

\begin{document}

\theoremstyle{plain}\newtheorem{teo}{Theorem}[section]
\theoremstyle{plain}\newtheorem{prop}[teo]{Proposition}
\theoremstyle{plain}\newtheorem{lem}[teo]{Lemma}
\theoremstyle{plain}\newtheorem{cor}[teo]{Corollary}
\theoremstyle{definition}\newtheorem{defin}[teo]{Definition}
\theoremstyle{remark}\newtheorem{rem}[teo]{Remark}
\theoremstyle{plain} \newtheorem{assum}[teo]{Assumption}
\theoremstyle{definition}\newtheorem{example}[teo]{Example}

\begin{abstract}
In this work we investigate the $c_0$-extension property. This property generalizes Sobczyk's theorem in the context of nonseparable Banach spaces. We prove that a sufficient condition for a Banach space to have this property is that its closed dual unit ball is weak-star monolithic. We also present several results about the $c_0$-extension property in the context of $C(K)$ Banach spaces. An interesting result in the realm of $C(K)$ spaces is that the existence of a Corson compactum $K$ such that $C(K)$ does not have the $c_0$-extension property is independent from the axioms of $ZFC$.
\end{abstract}

\maketitle

\begin{section}{Introduction}

Sobczyk's theorem \cite{Sobczyk} is a classical result about the structure of separable Banach spaces. It states that if $X$ is a separable Banach space, then every $c_0$-valued bounded operator defined on a closed subspace of $X$ admits a $c_0$-valued bounded extension defined on $X$. The search for generalizations of Sobczyk's theorem in the context of nonseparable Banach spaces has attracted a lot of attention in the last decades \cite{SobczykLine, c0EP, c0EPline, DrygierPlebanek, Eloi, Molto, Patterson}. In \cite{SobczykLine}, the $c_0$-extension property was introduced. We say that a Banach space $X$ has the {\it $c_0$-extension property} (\EP) if every $c_0$-valued bounded operator defined on a closed subspace of $X$ admits a $c_0$-valued bounded extension defined on $X$. Clearly Sobczyk's theorem implies that every separable Banach space has the \EP. Let us recall the main results known about the \EP. An adaptation of Veech's proof of Sobczyk's theorem \cite{Veech} shows that every weakly compactly generated Banach space has the \EP\ \cite[Proposition~2.2]{c0EP}. Recall that a Banach space is said to be {\it weakly compactly generated} (WCG) if it contains a weakly compact subset that is linearly dense. Clearly every separable and every reflexive Banach space is WCG and the canonical example of a neither separable nor reflexive WCG space is given by $c_0(I)$, for any uncountable set $I$ \cite[page 575]{Fabian}.
In \cite[Theorem~2.2]{SobczykLine}, it was shown that if $K$ is compact line, then $C(K)$ has the \EP\ if and only if $K$ is monolithic. Here, as usual, $C(K)$ denotes the Banach space of real-valued continuous functions defined on a compact Hausdorff space $K$, endowed with the supremum norm. Recall that a {\it compact line} is a totally ordered set that is compact when endowed with the order topology. The notion of monolithicity plays a central role in this work. We say that a compact Hausdorff space $K$ is {\it monolithic} if every separable subspace of $K$ is second countable.

One of the main results of this work is Theorem \ref{monoc0EP}, where we show that if $X$ is a Banach space with closed dual unit ball weak-star monolithic, then $X$ has the \EP.
We denote the closed dual unit ball of a Banach space $X$ by $B_{X^*}$ and the weak-star topology by $w^*$-topology. Observe that all the spaces with the \EP\ mentioned above have closed dual unit ball $w^*$-monolithic. Indeed, it is well-known that if $X$ is a separable Banach space, then $B_{X^*}$ is $w^*$-metrizable \cite[Proposition 3.103]{Fabian} and therefore $B_{X^*}$ is $w^*$-monolithic. Moreover, if $X$ is a WCG Banach space, then $(B_{X^*}, w^*)$ is an Eberlein compactum \cite[Theorem~13.20]{Fabian} and therefore $B_{X^*}$ is $w^*$-monolithic. Recall that a compact space is an {\it Eberlein compactum} if it is homeomorphic to a weakly compact subset of a Banach space, endowed with the weak topology. Finally it follows from Lemma \ref{MonoMBolaMono} that if $K$ is a monolithic compact line, then $B_{C(K)^*}$ is $w^*$-monolithic.
A really interesting class of monolithic compact spaces is the class of Corson compacta. We say that a compact space $K$ is a {\it Corson compactum} if there exists a set $I$ such that $K$ is homeomorphic to a closed subset of $\Sigma(I)$, endowed with the product topology. By $\Sigma(I)$ we denote the subset of $\R^I$ formed by functions with countable support.
Since every Corson compactum is monolithic, it follows from Theorem \ref{monoc0EP} that a Banach space $X$ has the \EP\ whenever $(B_{X^*}, w^*)$ is a Corson compactum. Those are precisely the {\it weakly Lindelöf determined} Banach spaces (WLD).
Therefore the class of Banach spaces with the \EP\ contains the class of WLD spaces. Note that there exist Banach spaces with the \EP\ that are not WLD. For instance, the space of continuous functions $C[0,\omega_1]$ defined on the ordinal segment $[0,\omega_1]$ has the \EP, since $[0,\omega_1]$ is a monolithic compact line, but $C[0,\omega_1]$ is not WLD. To see that $C[0,\omega_1]$ is not WLD, note that a necessary condition for a $C(K)$ space to be WLD is that $K$ is a Corson compactum and $[0,\omega_1]$ is not a Corson compactum. A classical class of Banach spaces that contains properly the WLD spaces is the class of $1$-Plichko spaces. We say that a Banach space $X$ is {\it $1$-Plichko} if $X^*$ contains a $1$-norming $\Sigma$-subspace (see \cite{Biorthogonal} for more information on those spaces). Recall that a Banach space $X$ is WLD if and only if $X^*$ is a $\Sigma$-subspace of itself \cite[Theorem~5.37]{Biorthogonal} and therefore every WLD space is $1$-Plichko. An example of a non WLD space that is $1$-Plichko is given by the space $\ell_1(I)$, for any uncountable set $I$ \cite[Example~6.9]{Kalenda}. A natural question at this point is whether every $1$-Plichko space has the \EP. In Proposition \ref{l1SemEP}, we answer this question negatively by showing that if $I$ is an uncountable set, then the space $\ell_1(I)$ does not have the \EP.

In Section \ref{C(K)spaces}, we investigate the \EP\ in the context of $C(K)$ spaces. It is well-known that if $K$ is a compact Hausdorff space, then $C(K)$ is separable if and only if $K$ is metrizable \cite[Lemma~3.102]{Fabian}. Therefore, it follows from Sobczyk's theorem that $C(K)$ has the \EP, for every metrizable compact space $K$. This result can be extended to the class of Eberlein compacta that generalizes the class of metrizable compact spaces. Recall that if $K$ is a compact Hausdorff space, then $K$ is an Eberlein compactum if and only if $C(K)$ is WCG \cite[Theorem~14.9]{Fabian}. Thus $C(K)$ has the \EP\ whenever $K$ is an Eberlein compactum.
Note that the class of Corson compacta generalizes properly the class of Eberlein compacta. In this work we address the problem of determining whether $C(K)$ has the \EP, for every Corson compactum. In Theorem \ref{CorsonIndependent}, we establish that the existence of a Corson compactum $K$ such that $C(K)$ does not have the \EP\ is independent from the axioms of ZFC. Another interesting result presented here is the characterization of the \EP\ for spaces of continuous functions on scattered compact spaces with small height. More precisely, in Theorem \ref{Scatteredc0EPiiMono} we show that if $K$ is a scattered compact space with height at most $\omega+1$, then $C(K)$ has the \EP\ if and only if $K$ is monolithic. It is worthwhile to compare this result with a similar characterization of the \EP\ in the context of compact lines obtained in \cite[Theorem~2.2]{SobczykLine}.
\end{section}

\begin{section}{General Results}
\label{sec:GeneralResults}

In this work we consider only real Banach spaces. An apparently stronger property than the \EP\ is the $c_0$-extension property with some constant.

\begin{defin}
Let $X$ be a Banach space and $\lambda \ge 1$ be a real number. We say that $X$ has the {\it $c_0$-extension property with constant $\lambda$} ($\lambda$-\EP) if any bounded operator $T:Y \to c_0$, defined on a closed subspace $Y$ of $X$, admits a bounded extension $\tilde{T}: X \to c_0$ with $\Vert \tilde{T} \Vert \le \lambda \Vert T \Vert$.
\end{defin}

We observe that all known proofs of \EP\ are actually proofs of $2$-\EP. It was shown in \cite[Proposition~2.2]{c0EP} that every WCG Banach space has the $2$-\EP. Moreover, the proof of \cite[Theorem~2.2]{SobczykLine} establishes that if $K$ is a monolithic compact line, then $C(K)$ has the $2$-\EP.

In what follows we devote ourselves to the proof of Theorem \ref{monoc0EP} which states that if $X$ is a Banach space such that $B_{X^*}$ is $w^*$-monolithic, then $X$ has the $2$-\EP. This proof relies essentially on Proposition \ref{monoBallequiv} that establishes an interesting characterization of the monolithicity of the dual unit ball of a Banach space in terms of $\ell_\infty$-valued bounded operators. The key ingredient in the proof of Proposition \ref{monoBallequiv} is a translation of monolithicity in terms of absolute bipolar sets presented in Lemma \ref{monoBipolar}. Let us recall some terminology and facts.
Given a real normed space $X$, there is a natural bilinear pairing of $X$ and $X^*$ given by:
\[X \times X^* \ni (x, \alpha) \to \alpha(x) \in \R.\]
Given a subset $A$ of $X^*$ the {\it absolute polar} $A^0$ of $A$, with respect to this bilinear pairing, is defined as:
\[A^0=\{x \in X: \vert \alpha(x) \vert \le 1, \ \forall \alpha \in A\}\]
and given a subset $B$ of $X$ the {\it absolute polar} $B^0$ of $B$, with respect to this bilinear pairing, is defined as:
\[B^0=\{\alpha \in X^*: \vert \alpha(x) \vert \le 1, \ \forall x \in B\}.\]
If $A$ is a subset of $X^*$, then the {\it absolute bipolar} $A^{00}$ of $A$ is defined as $A^{00}=(A^0)^0$.
Here we are considering only this natural bilinear pairing.
The proof of Lemma \ref{monoBipolar} relies on the deep relationship between the absolutely convex hull of subsets of $X^*$ and their absolute bipolars that is given by the {\em Absolute bipolar theorem} \cite[Theorem~3.1.1]{Bogachev}. For this particular bilinear pairing, the Absolute bipolar theorem ensures that $A^{00}$ coincides with the $w^*$-closure of the absolutely convex hull of $A$, for every subset $A$ of $X^*$.
Recall that if $A$ is a subset of $X^*$, then the {\it absolutely convex hull of $A$}, denoted here by $\absconv(A)$, is the smallest absolutely convex subset of $X^*$ containing $A$. It is easy to see that:
\[\absconv(A)=\big\{\sum_{i=1}^{n} t_i a_i: n \ge 1, a_i \in A, t_i \in \R \ \text{and $\sum_{i=1}^{n} \vert t_i \vert \le 1$}\big\}.\]

\begin{lem}\label{monoBipolar}
Let $X$ be a Banach space. The following conditions are equivalent:
\begin{enumerate}
\item[(i)] $B_{X^*}$ is $w^*$-monolithic.

\item[(ii)] $\overline{\absconv(A)}^{w^*}$ is $w^*$-metrizable, for every countable subset $A$ of $B_{X^*}$.

\item[(iii)] $A^{00}$ is $w^*$-metrizable, for every countable subset $A$ of $B_{X^*}$.

\end{enumerate}
\end{lem}
\begin{proof}
The equivalence between (ii) and (iii) follows directly from the Absolute bipolar theorem.
Observe that Uryshon's metrization theorem implies that a compact space $K$ is monolithic if and only if the closure of any countable subset of $K$ is metrizable.
Clearly (ii) implies (i). To see that (i) implies (ii), note that if $A$ is a countable subset of $X^*$, then $\overline{\absconv(A)}^{w^*}$ is $w^*$-separable. This separability follows from the fact that the countable set
\[\big\{\sum_{i=1}^{n} q_i a_i: n \ge 1, a_i \in A, q_i \in \mathbb Q \ \text{and $\sum_{i=1}^{n} \vert q_i \vert \le 1$}\big\}\]
is $w^*$-dense in $\overline{\absconv(A)}^{w^*}$.
\end{proof}

Let $X$ be a real normed space and $A$ be a bounded subset of $X^*$. We denote by $\rho_A: X \to [0,+\infty[$ the seminorn defined as:
\[\rho_A(x)=\sup_{\alpha \in A}\vert \alpha (x) \vert, \ \forall x \in X.\]

\begin{lem}\label{seminormSep}
Let $X$ be a Banach space. The following conditions are equivalent:
\begin{enumerate}
\item[(a)] $B_{X^*}$ is $w^*$-monolithic.

\item[(b)] $(X, \rho_{A})$ is separable, for every countable subset $A$ of $B_{X^*}$.
\end{enumerate}
\end{lem}
\begin{proof}
To prove that $(a)$ and $(b)$ are equivalent it suffices to show the equivalence between $(b)$ and condition $(iii)$ of Lemma \ref{monoBipolar}. Towards this goal, we will show that $A^{00}$ coincides with the closed unit ball of the dual space $(X,\rho_A)^*$, for any subset $A$ of $B_{X^*}$. The result will then follow from the fact that $(X, \rho_A)$ is separable if and only if the closed unit ball of its dual space is $w^*$-metrizable \cite[Proposition~3.103]{Fabian}. Let $A$ be a subset of $B_{X^*}$ and note that:
\[A^{00}=\{\alpha \in X^*: \vert \alpha(x) \vert \le 1, \ \forall x \in B_{(X, \rho_A)}\},\]
where $B_{(X, \rho_A)}$ denotes the closed unit ball of $(X, \rho_A)$.
Finally, to conclude the result just observe that $A^{00} \subset (X, \rho_A)^*$ and that $(X, \rho_A)^* \subset X^*$.
\end{proof}

Recall that if $X$ is a real normed space, then there exists a bijective correspondence between bounded operators $X \to \ell_\infty$ and bounded sequences in $X^*$.
More precisely, if $S:X \to \ell_\infty$ is a bounded operator, then the {\it sequence associated to $S$} is $(\pi_n \circ S)_{n \ge 1}$, where $\pi_n: \ell_\infty \to \R$ denotes the $n^{th}$ projection. Conversely if $(\alpha_n)_{n \ge 1}$ is a bounded sequence in $X^*$, then the {\it operator associated} to this sequence is defined as $S(x)=\big(\alpha_n(x)\big)_{n \ge 1}$, for every $x \in X$. Note that if $(\alpha_n)_{n \ge 1}$ is the sequence associated to a bounded operator $S:X \to \ell_\infty$, then $\Vert S \Vert=\sup_{n \ge 1} \Vert \alpha_n \Vert$. It follows from this correspondence and Hahn-Banach's theorem that every bounded $\ell_\infty$-valued operator defined on a subspace of $X$ admits a bounded $\ell_\infty$-valued extension defined on $X$ with the same norm. Any of those extensions is called a {\it Hahn-Banach extension} of the original operator. Given an operator $S$, we denote its image by $\Img S$.

\begin{prop}\label{monoBallequiv}
Let $X$ be a Banach space. The following conditions are equivalent:
\begin{enumerate}
\item $B_{X^*}$ is $w^*$-monolithic.

\item $\Img S$ is separable, for every bounded operator $S:X \to \ell_\infty$.
\end{enumerate}
\end{prop}
\begin{proof}
Assume $(1)$ and let $S:X \to \ell_\infty$ be a bounded operator. Without loss of generality, we may assume that $\Vert S \Vert=1$. Denote by $(\alpha_n)_{n \ge 1}$ the sequence associated to $S$ and set $A=\{\alpha_n: n \ge 1\}$. Since $A \subset B_{X^*}$, Lemma \ref{seminormSep} ensures that $(X,\rho_A)$ is separable. The separability of $\Img S$ follows from the fact that the onto operator $S:(X, \rho_A) \to \Img S$ is bounded and therefore continuous. Now let us prove that $(2)$ implies condition (b) of Lemma \ref{seminormSep}. Let $A$ be a countable subset of $B_{X^*}$ and enumerate $A=\{\alpha_n: n \ge1\}$. Denote by $S:X \to \ell_\infty$ the bounded operator associated to the sequence $(\alpha_n)_{n \ge 1}$. Let $E$ be a countable and dense subset of $\Img S$ given by $(2)$ and denote by $D$ a countable subset of $X$ such that $S[D]=E$. It is easy to see that $D$ is dense in $(X, \rho_A)$. Indeed, fixed $x \in X$ and $\epsilon>0$, the density of $E$ in $\Img S$ implies that there exists $e \in E$ with $\Vert S(x)-e \Vert< \epsilon$. If $d \in D$ satisfies $S(d)=e$, then:
\[\rho_A(x-d)=\sup_{n \ge 1} \vert \alpha_n(x-d) \vert=\Vert S(x-d) \Vert=\Vert S(x)-e \Vert< \epsilon.\]

\end{proof}

\begin{teo}\label{monoc0EP}
Let $X$ be a Banach space. If $B_{X^*}$ is $w^*$-monolithic, then $X$ has the $2$-\EP.
\end{teo}
\begin{proof}
Let $Y$ be a closed subspace of $X$ and $T:Y \to c_0$ be a bounded operator. If $S: X \to \ell_\infty$ denotes a Hahn-Banach extension of $T$, then Proposition \ref{monoBallequiv} ensures that $\Img S$ is a separable subspace of $\ell_\infty$. Therefore $\overline{\Img T}$ is a closed subspace of the separable Banach space $\overline{\Img S}$. It follows from Sobczyk's theorem that the inclusion map $i: \overline{\Img T} \to c_0$ admits a bounded extension $L: \overline{\Img S} \to c_0$ with $\Vert L \Vert \le 2 \Vert i \Vert \le 2$. The map $L \circ S: X \to c_0$ is a bounded extension of $T$ and $\Vert L \circ S \Vert \le \Vert S \Vert \Vert L \Vert\le 2 \Vert T\Vert$.
\end{proof}

We do not know if the converse of Theorem \ref{monoc0EP} holds.

\begin{cor}
If $X$ is a WLD Banach space, then $X$ has the $2$-\EP.
\end{cor}
\begin{proof}
It follows from Theorem \ref{monoc0EP} and the fact that every Corson compactum is monolithic.
\end{proof}

Now let us see some stability properties of the class of Banach spaces with the \EP.

\begin{prop}\label{fechamentosEP}
Let $X$ and $Z$ be Banach spaces and assume that $X$ has the \EP.
\begin{enumerate}
\item[(a)] If $Z$ is a subspace of $X$, then $Z$ has the \EP.

\item[(b)] If $Q:X \to Z$ is a bounded and onto operator, then $Z$ has the \EP.
\end{enumerate}
\end{prop}
\begin{proof}
The proof of (a) is straightforward. To prove (b), let $Y$ be a closed subspace of $Z$ and let $T:Y \to c_0$ be a bounded operator. Since $X$ has the \EP, the bounded operator $T \circ Q|_{Q^{-1}[Y]}$ admits a bounded extension $S: X \to c_0$. Define $\tilde{T}: Z \to c_0$ as $\tilde{T}(z)=S(x)$, where $x$ is any element of $X$ satisfying $Q(x)=z$. The fact that $\tilde{T}$ is well-defined follows from the fact that $\Ker Q$ is contained in $Q^{-1}[Y]$. Indeed, if $x_1$ and $x_2$ are elements of $X$ with $Q(x_1)=Q(x_2)$, then
$x_1-x_2 \in \Ker Q \subset Q^{-1}[Y]$. Therefore, we have:
\[S(x_1-x_2)=T(Q(x_1-x_2))=0.\]
It is easy to see that the map $\tilde{T}$ is a bounded extension of $T$.
\end{proof}

Next proposition states that the $1$-Plichko space $\ell_1(I)$ does not have the \EP, for any uncountable set $I$. We present its proof after Remark \ref{valdiviaSemEP},
since it depends on results developed in Section \ref{C(K)spaces}.

\begin{prop}\label{l1SemEP}
If $I$ is an uncountable set, then $\ell_1(I)$ does not have the \EP.
\end{prop}

\begin{cor}\label{copial1SemEP}
Let $X$ be a Banach space and $I$ be an uncountable set. If $X$ contains an isomorphic copy of $\ell_1(I)$, then $X$ does not have the \EP.
\end{cor}
\begin{proof}
It follows from Propositions \ref{fechamentosEP}(a) and \ref{l1SemEP}.
\end{proof}

It is interesting to observe that even though not every $1$-Plichko space has the \EP, these spaces have a weaker property that was also introduced in \cite{SobczykLine}; namely the separable $c_0$-extension property. We say that a Banach space $X$ has the {\it separable $c_0$-extension property} (separable \EP) if every $c_0$-valued bounded operator defined on a closed and separable subspace of $X$ admits a $c_0$-valued bounded extension defined on $X$. The fact that every $1$-Plichko space has the separable \EP\ follows from the fact that such spaces have the separable complementation property (SCP) \cite[pg. 105]{Biorthogonal} and clearly the SCP implies the separable \EP. Note that Proposition \ref{l1SemEP} provides an example of a space with the SCP but not the \EP. A stronger property than the SCP is the controlled separable complementation property (CSCP) that was introduced in \cite{Woj} and studied in a series of papers \cite{Ferrer1, Ferrer2, Ferrer3}. A surprising consequence of Theorem \ref{monoc0EP} is that the CSCP implies the \EP.

\begin{prop}
Let $X$ be a Banach space. If $X$ has the CSCP, then $X$ has the $2$-\EP.
\end{prop}
\begin{proof}
It follows from Theorem \ref{monoc0EP} and the fact that if $X$ has the CSCP, then $B_{X^*}$ is $w^*$-monolithic \cite[Proposition~1.5]{KK}.
\end{proof}

\end{section}

\begin{section}{The \EP\ for $C(K)$ spaces}
\label{C(K)spaces}

In this section, we develop the theory of the \EP\ in the context of Banach spaces of the form $C(K)$. Throughout this section we will refer to a compact and Hausdorff space just as a compact space. For a compact space $K$, we identify the dual space $C(K)^*$ with the space $M(K)$ of regular signed finite Borel measures on $K$, endowed with the total variation norm. In Lemma \ref{MonoMBolaMono}, we establish conditions on a compact space $K$ that ensure that $B_{M(K)}$ is $w^*$-monolithic. Clearly if $B_{M(K)}$ is $w^*$-monolithic, then $K$ is monolithic, since $K$ is homeomorphic to a closed subspace of $\big(B_{M(K)}, w^*\big)$. The additional condition on monolithic compacta used in Lemma \ref{MonoMBolaMono} is Property (M). We say that a compact space $K$ has {\it Property (M)} if the support of every measure belonging to $M(K)$ is separable. Recall that if $\mu$ is an element of $M(K)$, then its {\it support} is defined as:
\[\supp \mu=\{p \in K: \vert \mu \vert(U)>0, \ \text{for every open neighborhood $U$ of $p$}\},\]
where $\vert \mu \vert$ denotes the total variation of $\mu$.

\begin{lem}\label{MonoMBolaMono}
If $K$ is monolithic compact space with Property (M), then $B_{M(K)}$ is $w^*$-monolithic.
\end{lem}
\begin{proof}
Let $\{\mu_n: n \ge 1\}$ be a countable subset of $B_{M(K)}$ and denote by $L$ the closed subset of $K$ defined as $L=\overline{\bigcup_{n \ge 1} \supp \mu_n}$. Since $K$ has Property (M) and is monolithic, we have that $L$ is a compact metric space and therefore $B_{M(L)}$ is $w^*$-metrizable. To show that $\overline{\{\mu_n: n \ge 1\}}^{w^*}$ is $w^*$-metrizable, we will prove that this space is a subspace of an homeomorphic image of $(B_{M(L)}, w^*)$. Denote by $i_{*}: M(L) \to M(K)$ the operator defined as $i_*(\nu)(B)=\nu(B \cap L)$, for every Borel subset $B$ of $K$ and every $\nu \in M(L)$. Note that $i_*$ is injective and $w^*$-continuous. Therefore the $w^*$-compactness of $B_{M(L)}$ ensures that $\big(i_*[B_{M(L)}], w^*\big)$ is homeomorphic to $(B_{M(L)}, w^*)$.
Finally, observe that if $\mu \in B_{M(K)}$ and $\supp \mu \subset L$, then $\mu \in i_*[B_{M(L)}]$. This implies that $\overline{\{\mu_n: n \ge 1\}}^{w^*} \subset i_*[B_{M(L)}]$.
\end{proof}

\begin{teo}\label{monoMEP}
If $K$ is a monolithic compact space with Property (M), then $C(K)$ has the $2$-\EP.
\end{teo}
\begin{proof}
It follows from Theorem \ref{monoc0EP} and Lemma \ref{MonoMBolaMono}.
\end{proof}

Clearly every metrizable compact space has Property (M) and it is easy to see that every scattered compact space has Property (M). We say that a topological space $X$ is {\it scattered} if for every closed subspace $Y$ of $X$, the set of isolated points of $Y$ is dense in $Y$. Other examples of spaces with Property (M) are compact lines \cite[Lemma~2.1]{KK}, Eberlein compacta and Rosenthal compacta \cite[Remark~3.2]{ArgyrosCorson}. Recall that a {\it Rosenthal compactum} is a pointwise compact subset of the space of Baire class one real-valued functions on a Polish space.

\begin{cor}\label{ghat}
If $K$ is a monolithic compact space belonging to any of the following classes, then $C(K)$ has the $2$-\EP:
\begin{enumerate}
\item[(a)] Compact lines.

\item[(b)] Scattered spaces.

\item[(c)] Rosenthal compacta.
\end{enumerate}
\end{cor}
\begin{proof}
It follows from Theorem \ref{monoMEP} and the fact that those spaces have Property (M).
\end{proof}

It is well-known that the existence of a Corson compactum without Property (M) is independent from the axioms of $ZFC$. It was shown in \cite[Remark~3.2.3]{ArgyrosCorson} that if we assume Martin's axiom and the negation of the Continuum hypothesis ($MA+\neg CH$), then every Corson compactum has Property (M). Therefore Theorem \ref{monoMEP} implies that, under $MA+\neg CH$, $C(K)$ has the $2$-\EP, for every Corson compactum $K$. On the other hand, if we assume $CH$, then there exists a Corson compactum without Property (M) \cite[Theorem~3.12]{ArgyrosCorson}. It turns out that the space of continuous functions on the Corson compactum constructed in \cite[Theorem~3.12]{ArgyrosCorson} does not have the \EP.

\begin{prop}\label{CorsonSemEP}
Assume $CH$. There exists a Corson compactum $K$ such that $C(K)$ does not have the \EP.
\end{prop}
\begin{proof}
Let $K$ denote the Corson compactum without Property (M) constructed in \cite[Theorem~3.12]{ArgyrosCorson}. It follows from \cite[Theorem~3.13]{ArgyrosCorson} that $C(K)$ contains an isomorphic copy of $\ell_1(\omega_1)$. Therefore the result follows from Corollary \ref{copial1SemEP}.
\end{proof}

From what was discussed above we conclude the following result.

\begin{teo}\label{CorsonIndependent}
The existence of a Corson compactum $K$ such that $C(K)$ does not have the \EP\ is independent from the axioms of ZFC.
\end{teo}

\begin{rem}
The Corson compactum without Property (M) constructed in \cite[Theorem~3.12]{ArgyrosCorson} is defined by an adequate family of sets. In \cite[Theorem~3.13]{ArgyrosCorson}, it was shown that if $K$ is a Corson compactum defined by an adequate family of sets, then $K$ has Property (M) if and only if $\ell_1(\omega_1)$ does not embed isomorphically in $C(K)$. Thus Corollary \ref{copial1SemEP} and Theorem \ref{monoMEP} imply that if $K$ is a Corson compactum defined by an adequate family of sets, then $C(K)$ has the \EP\ if and only if $K$ has Property (M). We do not know if this equivalence holds for any Corson compactum. For instance it is unclear if $C(K)$ has the \EP\ for the Corson compactum $K$ without Property (M) constructed in \cite{Kunen}.
\end{rem}

Note that Proposition \ref{CorsonSemEP} also shows that if we assume $CH$, then there exists a monolithic compact space $K$ such that $C(K)$ does not have the \EP. However this is not possible under $MA+\neg CH$.

\begin{prop}\label{Martinmonotemc0EPsempre}
Assume $MA+\neg CH$. If $K$ is a monolithic compact space, then $C(K)$ has the $2$-\EP.
\end{prop}
\begin{proof}
It was shown in \cite[Theorem~3]{Arkangelski} that if we assume $MA+\neg CH$, then every monolithic compact space with the countable chain condition is separable. It is easy to see that this implies that, under $MA+\neg CH$, every monolithic compact space has Property (M). Thus the result follows from Theorem \ref{monoMEP}.
\end{proof}

From what was discussed above we conclude the following result.

\begin{teo}
The existence of a monolithic compact space $K$ such that $C(K)$ does not have the \EP\ is independent from the axioms of $ZFC$.
\end{teo}

Now we present some useful stability results for the \EP\ in the context of $C(K)$ spaces.

\begin{prop}\label{FecCompactEP}
Let $K$ and $L$ be compact spaces and assume that $C(K)$ has the \EP.
\begin{enumerate}
\item If $L$ is a subspace of $K$, then $C(L)$ has the \EP.

\item If $L$ is a quotient of $K$, then $C(L)$ has the \EP.
\end{enumerate}
\end{prop}
\begin{proof}
To prove (1), note that the restriction operator $C(K) \to C(L)$ is a bounded and onto operator. Therefore the result follows from Proposition \ref{fechamentosEP}(b).
Now let us prove (2). It is easy to see that if $\phi: K \to L$ is a continuous and onto map, then the composition operator $\phi^*: C(L) \to C(K)$ defined as $\phi^*(f)=f \circ \phi$, for all $f \in C(L)$, is an isometric embedding. Proposition \ref{fechamentosEP}(a) ensures that $\phi^*[C(L)]$ has the \EP\ and therefore $C(L)$ has the \EP.
\end{proof}

In what follows, we investigate the \EP\ in the class of scattered compact spaces. Let us recall some basic definitions and facts. Given a topological space $X$ and an ordinal number $\alpha$, we denote by $X^{(\alpha)}$ the $\alpha^{th}$ Cantor--Bendixson derivative of $X$.
It is well-known that a topological space $X$ is scattered if and only if there exists an ordinal $\alpha$ such that $X^{(\alpha)}=\emptyset$ \cite[Proposition~17.8]{Koppelberg}. The {\it height} of a scattered topological space $X$, denoted here by $\height(X)$, is defined as the least ordinal $\alpha$ such that $X^{(\alpha)}=\emptyset$. Finally if the height of $X$ is a natural number, then we say that $X$ has {\it finite height}.
In Corollary \ref{ghat}(b), we proved that if $K$ is a monolithic compact scattered space, then $C(K)$ has the \EP. In Theorem \ref{finiteheightEPMONO} and Corollary \ref{htw1}, we will show that if $K$ is a scattered compact space with height at most $\omega+1$ such that $C(K)$ has the \EP, then $K$ is monolithic.

\begin{lem}\label{MrowkaSemEP}
Let $K$ be a scattered compact space with height $3$. If $K$ is separable and nonmetrizable, then $C(K)$ does not have the \EP.
\end{lem}
\begin{proof}
Firstly let us prove that we can assume that $K^{(2)}$ has only one point. Write $K^{(2)}=\{p_1,\ldots,p_k\}$, where $k=\vert K^{(2)}\vert$. Clearly, there exist disjoint clopen subsets $C_1,\ldots,C_k$ of $K$ such that $K=\bigcup_{i=1}^{k} C_i$ and $p_i \in C_i$, for every $i=1,\ldots,k$. It is easy to see that $C_i$ is separable and $C_i^{(2)}=\{p_i\}$, for every $i=1,\ldots,k$ and that there exists an index $i_0$ such that $C_{i_0}$ is nonmetrizable.
Thus Proposition \ref{FecCompactEP}(1) ensures that the result follows if we show that $C(C_{i_0})$ does not have the \EP. Note that the restriction operator $R: C(K) \to C(K^{(1)})$ provides the following short exact sequence:
\[0 \longrightarrow \Ker(R) \longrightarrow  C(K) \overset{R}\longrightarrow C(K^{(1)}) \longrightarrow 0.\]
It is well-known that $\Ker(R)$ is isomorphic to $c_0$ and that $C(K^{(1)})$ is isomorphic to $c_0(K^{(1)})$. Therefore we have an exact sequence of the form:
\begin{equation}\label{exactSequence}
0 \longrightarrow c_0 \longrightarrow  C(K) \longrightarrow c_0(K^{(1)}) \longrightarrow 0.
\end{equation}
In order to conclude that $C(K)$ does not have the \EP, we will show that the isomorphic copy of $c_0$ inside $C(K)$ given by \eqref{exactSequence} is not complemented in $C(K)$.
Assume by contradiction that this copy of $c_0$ is complemented in $C(K)$. In this case $C(K)$ is isomorphic to $c_0 \bigoplus c_0(K^{(1)})$ and therefore $C(K)$ is a WCG space. This implies that $K$ is an Eberlein compactum. But this is not possible, since $K$ is not monolithic.
\end{proof}

An interesting consequence of Lemma \ref{MrowkaSemEP} is that the \EP\ is not a three-space property. We say that a property $\mathcal P$ is a {\it three-space property} if every twisted sum of Banach spaces with property $\mathcal P$ also has property $\mathcal P$. Classical examples of three-space properties for Banach spaces are separability and reflexivity \cite{CastilloGonzalez}.

\begin{cor}\label{Not3Propc0EP}
The \EP\ is not a three-space property.
\end{cor}
\begin{proof}
Consider the exact sequence \eqref{exactSequence} described in the proof Lemma \ref{MrowkaSemEP}. The spaces $c_0$ and $c_0(K^{(1)})$  have the \EP, since they are WCG spaces, but $C(K)$ does not have the \EP.
\end{proof}

The strategy to prove Theorem \ref{finiteheightEPMONO} is to proceed by induction on the height of $K$. In order to decrease the height and use the induction hypothesis we will use a quotient of $K$. Given a compact space $K$ and a closed subset $F$ of $K$, we define the following equivalence relation on K:
\[x \sim_F y \Leftrightarrow x=y \ \text{or} \ x,y \in F.\]
Clearly the quotient space $K/\sim_F$ is compact and it is easy to see that it is also Hausdorff.

\begin{lem}\label{QuotientF}
Given an ordinal $\alpha$, if $F=K^{(\alpha)}$, then $\big(K/\sim_F\big)^{(\alpha+1)}=\emptyset$.
\end{lem}
\begin{proof}
It follows from the fact that if $K$ and $L$ are compact spaces and $q:K \to L$ is a continuous and onto map, then $L^{(\alpha)} \subset q[K^{(\alpha)}]$, for every ordinal $\alpha$ \cite[Lemma~1.8 and Theorem~1.9]{QuotientScattered}.
\end{proof}

\begin{teo}\label{finiteheightEPMONO}
Let $K$ be a compact space with finite height. If $C(K)$ has the \EP, then $K$ is monolithic.
\end{teo}
\begin{proof}
It follows from Proposition \ref{FecCompactEP}(1) that it suffices to prove that if $K$ is a separable compact space with finite height such that $C(K)$ has the \EP, then $K$ is metrizable. Let us proceed by induction on the height of $K$. The case when $\height(K)\le 2$ is trivial and the case when $\height(K)=3$ is established in Lemma \ref{MrowkaSemEP}.
Now assume that $\height(K)=N+1$ with $N \ge 3$ and denote by $F$ the closed subset $K^{(N-1)}$ of $K$. Note that Lemma \ref{QuotientF} implies that $K/\sim_F$ is a scattered space with $\height\big(K/\sim_F\big) \le N$. Since Proposition \ref{FecCompactEP}(2) ensures that $C\big(K/\sim_F\big)$ also has the \EP, it follows from the induction hypothesis that $K/\sim_F$ is metrizable. Recall that if an infinite compact space is scattered, then its weight coincides with its cardinality \cite[Proposition~17.10]{Koppelberg}. Thus we have that $K/\sim_F$ is countable and consequently $K \setminus F$ is countable. This implies that $K^{(1)} \setminus K^{(2)}$ is countable and therefore $K^{(1)}$ is separable.
Note that $\height(K^{(1)})=N$ and that Proposition \ref{FecCompactEP}(1) ensures that $C(K^{(1)})$ has the \EP. Thus it follows from the induction hypothesis that $K^{(1)}$ is metrizable and therefore $K^{(1)}$ is countable, since it is scattered.
Finally the metrizability of $K$ follows from the fact that $K \setminus K^{(1)}$ is also countable, since $K$ is separable.
\end{proof}

\begin{cor}\label{htw1}
Let $K$ be a compact space with height $\omega+1$. If $C(K)$ has the \EP, then $K$ is monolithic.
\end{cor}
\begin{proof}
It follows from Proposition \ref{FecCompactEP}(1) that it suffices to prove that if $K$ is a separable compact space with height $\omega+1$ such that $C(K)$ has the \EP, then $K$ is metrizable. In order to do so, we will prove that $K \setminus K^{(N+1)}$ is countable, for every $N \in \omega$. This will imply that $K$ is countable and therefore metrizable, since $K=K^{(\omega)} \cup \bigcup_{N \in \omega} K \setminus K^{(N+1)}$ and $K^{(\omega)}$ is finite. Given $N \in \omega$, denote by $F$ the closed subset $K^{(N+1)}$ of $K$ and note that Lemma \ref{QuotientF} implies that $K/\sim_F$ is a scattered compact space with $\height\big(K/\sim_F\big) \le N+2$. Since Proposition \ref{FecCompactEP}(2) ensures that $C\big(K/\sim_F\big)$ also has the \EP, it follows from Theorem \ref{finiteheightEPMONO} that $K/\sim_F$ is metrizable. This implies that $K/\sim_F$ is countable and thus $K \setminus F$ is countable.
\end{proof}

\begin{teo}\label{Scatteredc0EPiiMono}
If $K$ is a scattered compact space with height at most $\omega+1$, then $C(K)$ has the \EP\ if and only if $K$ is monolithic.
\end{teo}
\begin{proof}
It follows from Corollary \ref{ghat}(b), Theorem \ref{finiteheightEPMONO} and Corollary \ref{htw1}.
\end{proof}

An interesting consequence of Lemma \ref{MrowkaSemEP} is the existence of Valdivia compacta whose spaces of continuous functions do not have the \EP. Recall that a compact space $K$ is said to be a {\it Valdivia compactum} if there exists a set $I$ and an homeomorphic embedding $\varphi: K \to \R^I$ such that $\varphi^{-1}[\Sigma(I)]$ is dense in $K$. Clearly, every Corson compactum is Valdivia and it is easy to see that if $\kappa$ is an uncountable cardinal, then the Cantor cube $2^\kappa$ is a non-Corson Valdivia compactum \cite[Theorem~3.29]{Kalenda}.

\begin{prop}\label{2kappanoEP}
If $\kappa$ is an uncountable cardinal, then the space $C(2^\kappa)$ does not have the \EP.
\end{prop}
\begin{proof}
Proposition \ref{FecCompactEP}(1) ensures that it suffices to show that $C(2^{\omega_1})$ does not have the \EP, since $2^{\omega_1}$ embeds homeomorphically in $2^\kappa$, for any uncountable cardinal $\kappa$. Let $K$ be a separable scattered compact space with height $3$ and weight $\omega_1$.
Since $K$ is zero-dimensional, we have that $K$ embeds homeomorphically in $2^{\omega_1}$. Therefore the result follows from Proposition \ref{FecCompactEP}(1) and Lemma \ref{MrowkaSemEP}.
\end{proof}

\begin{rem}\label{valdiviaSemEP}
Recall that if $K$ is a Valdivia compactum, then $C(K)$ is a $1$-Plichko space \cite[Theorem~5.55]{Biorthogonal}. Thus it follows from what was discussed in the previous section that $C(K)$ has the separable \EP, for every Valdivia compactum $K$.
\end{rem}

\begin{proof}[Proof of Proposition \ref{l1SemEP}]
According to Proposition \ref{fechamentosEP}(b), it suffices to show that if $I$ is an uncountable set, then $\ell_1(I)$ has a Banach quotient without the \EP. It is well-known that every Banach space with density $\vert I \vert$ is a quotient of $\ell_1(I)$. The result follows from Proposition \ref{2kappanoEP}, since the density of $C(2^{I})$ is $\vert I \vert$.
\end{proof}

We can generalize Proposition \ref{2kappanoEP} to the class of dyadic compacta. A compact space is said to be a {\it dyadic compactum} if it is a continuous image of the Cantor cube $2^\kappa$, for some cardinal $\kappa$.

\begin{cor}
If $K$ is a nonmetrizable dyadic compactum, then $C(K)$ does not have the \EP.
\end{cor}
\begin{proof}
Gerlits and Efimov showed that every nonmetrizable dyadic compactum contains an homeomorphic copy of the Cantor cube $2^{\omega_1}$ (see \cite[3.12.12]{Engelking}). Therefore the result follows from Propositions \ref{FecCompactEP}(1) and \ref{2kappanoEP}.
\end{proof}

We conclude this section by showing that if $K$ is a compact space such that $C(K)$ has the \EP, then $K$ admits only measures with small Maharam type. Given a compact space $K$ and a cardinal $\kappa$, we say that a nonnegative measure $\mu \in M(K)$ has {\it Maharam type} $\kappa$ if $\kappa$ is the density of the Banach space $L_1(\mu)$.

\begin{prop}
Let $K$ be a compact space. If $C(K)$ has the \EP, then every nonnegative measure in $M(K)$ has Maharam type at most $\omega_1$.
\end{prop}
\begin{proof}
If there exists a nonnegative measure in $M(K)$ with Maharam type strictly greater than $\omega_1$, then $C(K)$ contains an isomorphic copy of $\ell_1(\omega_1)$ \cite{PlebanekMaharam}. Therefore Corollary \ref{copial1SemEP} implies that $C(K)$ does not have the \EP.
\end{proof}

Assuming $MA+\neg CH$, we obtain the following stronger result.

\begin{prop}
Assume $MA+\neg CH$. Let $K$ be a compact space. If $C(K)$ has the \EP, then every nonnegative measure in $M(K)$ has countable Maharam type.
\end{prop}
\begin{proof}
Under $MA+\neg CH$, if there exists a nonnegative measure in $M(K)$ with uncountable Maharam type, then $C(K)$ contains an isomorphic copy of $\ell_1(\omega_1)$ \cite{PlebanekMaharam}. Therefore, Corollary \ref{copial1SemEP} implies that $C(K)$ does not have the \EP.
\end{proof}

\end{section}

\end{document}